\documentclass[11pt]{article}
\usepackage{mathrsfs}
\usepackage{latexsym,lineno}
\usepackage{epsfig}
\usepackage{color}
\usepackage{amsmath}\usepackage{fleqn}\usepackage{verbatim}\usepackage{epsf}
\usepackage{amsthm}\usepackage{graphicx, float}\usepackage{graphicx}
\usepackage{amsfonts}\usepackage{amssymb}\usepackage{graphpap}
\usepackage{epic}\usepackage{curves}

\newcommand{\be}{\begin{equation}}
\newcommand{\ee}{\end{equation}}
\newcommand{\benum}{\begin{enumerate}}
\newcommand{\eenum}{\end{enumerate}}
\newcommand{\bit}{\begin{itemize}}
\newcommand{\eit}{\end{itemize}}
\newtheorem{thom}{Theorem}[section]
\newtheorem{lemma}{Lemma}[section]

\newtheorem{prop}{Proposition}[section]

\newtheorem{coro}{Corollary}

\usepackage{graphicx}

\begin{document}
\def\s{\subseteq}
\def\n{\noindent}
\def\se{\setminus}
\def\dia{\diamondsuit}
\def\la{\langle}
\def\ra{\rangle}


\title{Sharp upper bounds for  multiplicative Zagreb indices of bipartite graphs with  given diameter}

\author{Chunxiang Wang$^a$, ~ Jia-Bao Liu$^b$, ~Shaohui Wang$^{c,}$\footnote{  Authors'email address: C. Wang(e-mail: wcxiang@mail.ccnu.edu.cn),
J.B. Liu (e-mail:  liujiabaoad@163.com),
S. Wang (e-mail: shaohuiwang@yahoo.com).}\\
\small\emph {a. School of Mathematics and Statistics, Central China Normal University, Wuhan,
430079, P.R. China}\\
\small\emph {b. School of Mathematics and Physics, Anhui Jianzhu
University, Hefei 230601, P.R. China} \\
\small\emph {c.  Department of Mathematics and Computer Science, Adelphi University, Garden City, NY 11550, USA} }
\date{}
\maketitle

\begin{abstract}

The first multiplicative Zagreb index of  a graph $G$ is the product of the square of every vertex  degree, while the second multiplicative Zagreb index is the product of the degree of each edge over all edges.   In our work,  we explore  the multiplicative Zagreb indices of  bipartite graphs of order $n$ with diameter $d$, and sharp upper bounds are obtained for these indices of graphs in $\mathcal{B}(n,d)$, where $\mathcal{B}(n, d)$ is the set of all  $n$-vertex bipartite graphs with the diameter $d$. In addition, we explore the relationship between the maximal  multiplicative Zagreb indices of graphs \textcolor{blue}{within} $\mathcal{B}(n, d)$.  As   consequences, those bipartite graphs with the largest, second-largest and smallest  multiplicative  Zagreb indices are characterized, and our results extend and enrich  some known conclusions.
\\
Accepted by Discrete Applied Mathematics.

\end{abstract}

\section{Introduction}

In the interdisplinary area between chemistry  and   mathematics, molecular graph invariants or descriptors could be used in the study of quantitative structure-property relationships (QSPR) and quantitative structure-activity relationships (QSAR). It would be helpful for describing partially  biological and chemical
properties, including physico-chemical (boiling and melting points)  and biological properties(toxicity)~\cite{Gutman1996,0103,LiuPX2015}. Among the most significant molecular descriptors,  the classically   molecular invariant
is named as Zagreb indices~\cite{Gutman1972}, which are expressed as  expected formulas for the total
$\pi$-electron energy of conjugated molecules as follows.
\begin{eqnarray} \nonumber
M_1(G) = \sum_{u \in V(G)} d(u)^2,
~~
~ M_2(G) = \sum_{uv \in E(G)} d(u)d(v),
\end{eqnarray}
where  $G$ is a (molecular) graph, $uv$ is a bond between two atoms $u$ and $v$, and $d(u)$ (or $d(v)$, respectively) is the number of atoms that are connected with $u$ (or $v$, respectively).

Many researchers are attracted by the idea of finding bounds for graph invariants and the related problem of figuring out the graphs achieving
  the maximum and minimum values of corresponding indices~\cite{L2016,L2015,0101,0102}. Nowadays, there are lots of articles
related to Zagreb indices in the interdisplinary area between  chemistry and
mathematics~\cite{Li2008,shi2015,BF2014,SM2014,WangJ2015,Iranmanesh20102,Ramin2016,Bojana2015}.  For instance,
Borovi\'canin et al.~\cite{Borov2016} introduced
bounds on Zagreb indices of trees in terms of domination number
and extremal trees are characterized.   Guo et al. \cite{d01} provided  the values of $M_1(G)$  of bipartite graphs. Cheng et al. \cite{d02,d03} studied the upper and lower bounds for the first Zagreb index
with a given number of vertices and edges.
Considering the  successful applications on Zagreb indices~\cite{Gutman2014}, Todeschini et al.(2010)~\cite{RT20102,Wang2015} presented the
following multiplicative variants of molecular structure
descriptors:
\begin{eqnarray} \nonumber
 \prod_1(G) = \prod_{u \in V(G)} d(u)^2,  ~
~~
 \prod_2(G) = \prod_{uv \in E(G)} d(u)d(v) = \prod_{u \in V(G)}
d(u)^{d(u)}.
\end{eqnarray}

Recently, many researchers are {interested} in (mutiplicative) Zagreb indices of complicate graphs.
 Xu and Hua~\cite{Xu20102} provided a unified approach to
characterize extremal (maximal and minimal) trees, unicyclic
graphs and bicyclic graphs with respect to multiplicative Zagreb
indices, respectively. Liu and Zhang \cite{Liuz20102} investigated several sharp upper bounds for
$\prod_1$-index and $\prod_2$-index in terms of graph parameters
such as the order, size and radius~\cite{Liuz20102}. Wang and Wei~\cite{Wang2015} gave sharp upper and lower bounds of multiplicative Zagreb indices for
$k$-trees.  Feng et al. \cite{d04} determined the graphs with given diameter for Zagreb indices. Li and Zhang \cite{d05} gave the upper bounds for Zagreb indices of bipartite graphs with a given diameter.

Motivated by  above statements, in this paper we further
investigate sharp upper bounds  for  the multiplicative Zagreb indices  of graphs in $\mathcal{B}(n,d)$, which is the set of $n$-vertex bipartite graphs with
diameter $d$.  In addition, we explore the relationship between the maximal  multiplicative Zagreb indices of graphs among $\mathcal{B}(n, d)$.  As  consequences, these bipartite graphs with the largest, second-largest and smallest  mutiplicative  Zagreb indices are characterized.

\section{Preliminary}
Let $G$ be a  simple, connected and undirected  graph. Denote a graph by $G = (V, E)$ with the vertex set $V = V (G)$  and the edge set $E = E(G)$.
   $|G|$ is defined as the cardinality of $V(G)$. For a vertex $v \in V(G)$, the neighborhood of  $v $ is the set $N(v) = N_G(v) = \{w \in V(G), vw \in E(G)\}$, and $d_G(v)$ (or briefly $d(v)$)  denotes the degree of $v$ with $ d_{G}(v) = |N(v)|$.
For $S \subseteq V(G)$ and  $F \subseteq E(G)$,   we use $G[S]$ for the subgraph of $G$ induced by $S$, $G - S$ for the subgraph induced by $V(G) - S$ and $G - F$ for the subgraph of G obtained by deleting $F$.
 Let $\mathcal{B}(n, d)$ be the set of $n$-vertex bipartite graphs with diameter $d$.  We denote $\lfloor x\rfloor$  by the largest integer not greater than $x$ and $\lceil x \rceil$ by the smallest integer not less than $x$, where  $x$ is a real number.

Note that  if $d=1$,  the unique bipartite {graph in $\mathcal{B}_{n,d}$} is  $K_2$. So we  should assume that $d \geq 2$ in the whole work.  Clearly, there exists a partition $V_0, V_1, \cdots, V_d$ of $V(G)$ such that $V_0 = \{u\}$ and $V_i = \{v: d(u, v) = i,  v \in V(G)\}$ with $1 \leq i \leq d$. We {call} this partition  a $d$-partition.
Also, each $V_i$ ($0\leq i \leq d$) is called a partition set. Let $|V_j|=m_j$, $1\leq j\leq d$.
By  routine calculations, one can derive the following propositions.
\begin{prop}
Let $f(x) = \frac{x}{x+m}$ be a function with $m > 0$. Then $f(x)$ is increasing in $\mathbb{R}$.
\end{prop}

\begin{prop}
Let $g(x) = \frac{x^x}{(x+m)^{x+m}}$ be a function with $m > 0$. Then $g(x)$ is decreasing  in $\mathbb{R}$.
\end{prop}

\section{The graphs achieving the sharp upper bounds of  mutiplicative Zagreb indices}

We first introduce several lemmas, which are critical to deduce the sharp upper bounds of  mutiplicative Zagreb indices.

\begin{lemma}\cite{Zhai2009} Let $G$ be a graph in $\mathcal{B}$$(n, d)$ with a $d$-partition of $V(G)$. Then the induced graph $G[V_i]$ is an empty graph (that is, $G[V_i]$ has no edges) for all $i$ with $0 \leq i \leq d$.
\end{lemma}

Based on Propositions 2.1 and 2.2, we will continue to consider the properties and structures of the partition $V_0, V_1, V_2, \cdots, V_d$.
Then the following lemmas are derived.

\begin{lemma} Let $G \in $ $ \mathcal{B}(n, d)$ with the maximal $\prod_1$-value.  Then the induced graph $G[V_{i-1} \cup V_i]$ is a complete bipartite subgraph for   all $i$ with $1 \leq i \leq d$.  Furthermore,  if $d \geq 3$ then $|V_d|=1$; If $d=2$, then $|V_1|=|V_2|=\frac{n-1}{2}$ with $n$ is odd and  $|V_1|=\frac{n}{2}$, $|V_2|=\frac{n}{2}-1$ with $n$ is even.
\end{lemma}

\begin{proof}
By the concept of first multiplicative Zagreb index, the first part of this lemma is directly deduced, that is,
the induced graph $G[V_{i-1} \cup V_i]$ is a complete bipartite subgraph for   all $i$ with $1 \leq i \leq d$.

 Next,  we will prove the second part by a contradiction.
Let $d \geq 3$, $x \in  V_d$ and $y \in  V_{d-3}$. If $|V_d| \geq  2$, then $G+xy \in$ $  \mathcal{B}(n, d)$ and $V_0 \cup V_1 \cup \cdots \cup V_{d-3} \cup (V_{d-2} \cup \{x\})\cup V_{d-1} \cup (V_d -\{x\})$ is a partition of $G + xy$. By  routine calculations, we see that $\prod_1(G + xy) > \prod_1(G)$, which is a contradiction.

Finally, let $d =2 $. Suppose $|V_1|=s, |V_2|=t$, since $|V_0|=1$, $s+t+1=n$. If $t \geq  2$, then $G' \in$ $  \mathcal{B}(n, d)$ and $V_0' \cup V_1' \cup V_{2}' $ is a partition of $G'$ with $V_0'=V_0,   V_1'=V_1\cup \{x\}, V_{2}' =V_2-\{x\}$, where $x\in V_2$. Since  $G[V_{i-1}' \cup V_i']$ ($i=1,2$) is a complete bipartite subgraph,
by the routine calculations, we see that 
\begin{eqnarray}
 \nonumber  \frac{ \prod_1({G} )}{\prod_1({G'} )} &&=~~~~\frac{ s^2(t+1)^{2s}s^{2t}}{(s+1)^{2t}t^{2(s+1)}}\\
  \nonumber&& =~~~~\frac{(\frac{s}{s+1})^{2t}s^2 }{(\frac{t}{t+1})^{2s}t^2} \;\;\;\;(by \;\;Proposition \;\;2.1)\\
&&\left\{\begin{array}{lcl} 
 <1,& \mbox{if} &  s<t, \\
>1,~~~~~~~~~~& \mbox{if} & s>t,\\
=1,~~~~~~~~~~& \mbox{if} &s=t.
\end{array}\right.
 \end{eqnarray}

Thus, when $t\geq 2$, if $d=2$, then $s=t=\frac{n-1}{2}$ with odd  $n$  and  $s=\frac{n}{2}$, $t=\frac{n}{2}-1$ with  even  $n$.
When $t=1$ and $d=2$, we have $|V_0|=t=1, s=n-2$, then $\prod_1 = (n-2)^44^{(n-2)}$. By  routine calculations, we complete the proof of Lemma 3.2.
\end{proof}

By the same method of Lemma 3.2, the following lemma is obtained analogously.

\begin{lemma}
Let $G \in$  $\mathcal{B}$$(n, d)$ with the maximal $\prod_2$-value. Then the induced graph $G[V_{i-1}\cup V_i]$ is a complete bipartite subgraph for   all $i$ with $1 \leq i \leq d$,  and $|V_d|=1$ if $d \geq 3$;   If $d=2$, then $|V_1|=|V_2|=\frac{n-1}{2}$ with $n$ is odd and  $|V_1|=\frac{n}{2}$, $|V_2|=\frac{n}{2}-1$ with $n$ is even.
\end{lemma}
\begin{proof}

By the same method of Lemma 3.2, the induced graph $G[V_{i-1}\cup V_i]$ is a complete bipartite subgraph for   all $i$ with $1 \leq i \leq d$,  and $|V_d|=1$ if $d \geq 3$.

Let $d =2 $, and suppose $|V_1|=s, |V_2|=t$, since $|V_0|=1$, $s+t+1=n$. If $t \geq  2$, then $G' \in$ $  \mathcal{B}(n, d)$ and $V_0' \cup V_1' \cup V_{2}' $ is a partition of $G'$ with $V_0'=V_0,   V_1'=V_1\cup \{x\}, V_{2}' =V_2-\{x\}$, where $x\in V_2$. Since  $G[V_{i-1}' \cup V_i']$ ($i=1,2$) is a complete bipartite subgraph,
by  routine calculations, we see that
\begin{eqnarray}
 \nonumber  \frac{ \prod_2({G} )}{\prod_2({G'} )} &&=~~~~\frac{ s^s(t+1)^{s(t+1)}s^{st}}{(s+1)^{t(s+1)}t^{t(s+1)}}
\\  \nonumber &&=~~~~ \frac{ (s(t+1))^{s(t+1)}}{ (t(s+1))^{t(s+1)}}\\
  &&\left\{ \begin{array}{lcl}
 <1,& \mbox{if} &  s<t, \\
>1,~~~~~~~~~~& \mbox{if} & s>t,\\
=1,~~~~~~~~~~& \mbox{if} &s=t.
\end{array}\right.
\end{eqnarray}

Thus, when $t\geq 2$, if $d=2$, then $s=t=\frac{n-1}{2}$ with  odd $n$ and  $s=\frac{n}{2}$, $t=\frac{n}{2}-1$ with  even $n$.
When $t=1$ and $d=2$, we have $|V_0|=t=1, s=n-2$, then $\prod_2 = (n-2)^{2n-4}4^{(n-2)}$. By  routine calculations, we complete the proof of Lemma 3.2.
\end{proof}

\begin{lemma}  Let  $G \in$  $\mathcal{B}$$(n, d)$  with the maximal  $\prod_1$-value. Then there are the following results:

(i)There does not exist three partition sets $V_i$, $V_j $ and $V_k$ such that $|V_i| \geq 2$, $|V_j|\geq 2$ and $|V_k|\geq  2$ with
$0\leq i , j ,k \leq d$.

(ii)If there are two partition sets $V_i$ and $V_j$ such that $|V_i| \geq 2$ and $|V_j| \geq  2$, then $|i-j|=1$ with $0\leq i, j \leq d.$

\end{lemma}

\begin{proof}
For $d=2$, the proof is done. For $d=3$, by $|V_0|=|V_d|=1$, the proof is also done. Now we consider the case of  $d \geq 4$. If there exists just one partition set (say $V_i$)  in $G$ such that $|V_i| \geq2$, then we are done. In order to show this lemma, it is enough to prove that if there exist at least two partition sets whose orders are greater than or equal to $2$, then we can deduce that for each pair of such partition sets (say $V_i $ and $V_j$) with $|V_i| \geq 2$ and $|V_j| \geq 2$, it must be the case that $|i - j| =1$.  Note that by Lemma 3.2, $G[V_{\ell-1}\cup V_{\ell}]$  is a complete bipartite subgraph for each $\ell \in  \{1, 2,\cdots , d\}$. We can obtain that each vetex in partition set $V_{\ell}$  has the same degree (say $d_{\ell}) ~\ell = 1, \cdots , d$.

Choose a graph $G$ in  $\mathcal{B}(n, d)$ such that it is achieving the maximal value of the first multiplicative Zagreb index. Assume that there exist two partition sets $V_i$ and $V_j$ such that $|V_i| \geq  2, |V_j| \geq 2 $ and $|i-j|\geq 2$.
We consider two cases below.

\noindent{\bf \small Case 1. } $m_{i+1}=m_{i+2}=m_{i+3}=\cdots =m_{j-2}= m_{j-1}=1$.

In order to proceed conveniently, we set that

$ A = (m_{i-2}+m_i)^{m_{i-1}}(1+m_i)(1+m_j)(1+m_{j+1})(m_j + m_{j+2})^{m_{j+1}}$,

 $ B=(m_{i-2}+m_i +1)^{m_{i-1}}(2+m_i) m_j(1+m_{i-1})(m_j+m_{j+2}-1)^{m_{j+1}}$,

  $C=(1+m_j)(m_j+m_{j+2})^{m_{j+1}}(1+m_i)(1+m_{i-1})(m_i+m_{i-2})^{m_{i-1}}$,

  $D=(2+m_j)(m_j+m_{j+2}+1)^{m_{j+1}}(m_i)(1+m_{j+1})(m_i+m_{i-2}-1)^{m_{i-1}}$.

\noindent{\bf \small Subcase  1.1.   } $A < B$.

 We choose a vertex $u \in V_j$ and let $G '��$  be the graph obtained by deleting all edges incident to u and joining u to each vertex in $V_{i-1}\cup V_{i+1}$ of $G $. Clearly $G '�� \in \mathcal{B}$$(n, d)$. Hence we have

\begin{eqnarray}
 \nonumber  \frac{ \prod_1({G} )}{\prod_1({G'} )} &=&\frac{d_{i-1}^{2m_{i-1}}d_i^{2m_i}d_{i+1}^{2m_{i+1}}  d_{j-1}^{2m_{j-1}}d_j^{2m_j}d_{j+1}^{2m_{j+1}} }{ (d_{i-1}+1)^{2m_{i-1}}d_i^{2(m_i+1)}(d_{i+1}+1)^{2m_{i+1}}  (d_{j-1}-1)^{2m_{j-1}}d_j^{2(m_j-1)}(d_{j+1}-1)^{2m_{j+1}}}
\\  \nonumber &=&(\frac{m_{i-2}+m_i}{1+m_{i-2}+m_i})^{2m_{i-1}} (\frac{1+m_i}{2+m_i})^2 [\frac{(1+m_j)(1+m_{j+1})(m_j+m_{j+2})^{m_{j+1}}}{m_j(1+m_{i-1})(m_j+m_{j+2}-1)^{m_{j+1}}}]^2 \;\;(by \; (3))
  \\ \nonumber &=&\frac{A^2}{B^2}
 < 1.
\end{eqnarray}

 Thus,  $\prod_1(G') > \prod_1(G)$,  which is a contradiction to the assumption.

\noindent{\bf \small Subcase 1.2. } $A\geq B$.��

\noindent{\bf Claim 1.} If  $A\leq B$, then $D< C$.
\begin{proof}
 It is straightforward to check that $D< C$ since $A \leq B$.
\end{proof}

 By Claim 1, if $C\leq D$, then $A>B$. Thus, if $A\geq  B$, then $C<D$.    (Otherwise, if $A\geq B$, then $C>D$ $\Leftrightarrow$ if $C\leq D$, then $A\leq B$, a contradiction).

We choose a vertex $u \in V_i$ and let $G '��$  be the graph obtained by deleting all edges incident to u and joining u to each vertex in $V_{j-1}\cup V_{j+1}$ of $G $. Clearly $G '�� \in \mathcal{B}$$(n, d)$. Hence we have
\begin{eqnarray}
 \nonumber  \frac{ \prod_1({G} )}{\prod_1({G'} )} =\frac{C^2��}{D^2��}<1. ��
\end{eqnarray}

 Thus,  $\prod_1(G') > \prod_1(G)$,  which is a contradiction to the assumption.

\noindent{\bf \small Case 2.  }  $V_i$, $V_{i+1} $ and $V_{i+2}$  are successive three partitions such that  $|V_i|\geq 2$ , $|V_{i+1}|\geq 2 $ and $|V_{i+2}|\geq 2$, where $V_{j}=V_{i+2}$.  Similarly, we set that

      $A' = (m_{i-2}+m_i)^{m_{i-1}}(m_{i+1}+m_{j+1})(m_j + m_{j+2})^{m_{j+1}}$,

   $B'=(m_{i-2}+m_i +1)^{m_{i-1}}(m_{i-1}+m_{i+1})(m_j+m_{j+2}-1)^{m_{j+1}}$,

 $ C'=(m_j+m_{j+2})^{m_{j+1}}(m_{i+1}+m_{i-1})(m_i+m_{i-2})^{m_{i-1}}$,

  $D'=(m_j+m_{j+2}+1)^{m_{j+1}}(m_{i+1}+m_{j+1})(m_i+m_{i-2}-1)^{m_{i-1}}$.

\noindent{\bf \small Subcase 2.1. } $A'< B'$.

 We choose a vertex $u \in V_j$ and let $G '��$  be the graph obtained by deleting all edges incident to u and joining u to each vertex in $V_{i-1}\cup V_{i+1}$ of $G $. Clearly $G '�� \in \mathcal{B}$$(n, d)$. Hence we have
\begin{eqnarray}
 \nonumber  \frac{ \prod_1({G} )}{\prod_1({G'} )} =(\frac{A'}{B'})^2<1.
 \end{eqnarray}

\noindent{\bf \small Subcase  2.2 }  $A' \geq B' $.��

\noindent{\bf Claim 2��} If  $A'\leq B'$, then $D'< C'$.
\begin{proof}
It is straightforward to check that $D' < C'$ since $A' \leq B'$.
\end{proof}

 By Claim 2 , if $C' \leq D'$, then $A' > B'$. Thus, if $A' \geq  B' $, then $C' < D' $.   (Otherwise, if $A' \geq  B'$, then $C'>D'$ $\Leftrightarrow$ if $C' \leq D' $, then $A' < B'$, a contradiction). Thus, if $A'\geq  B'$,   then $C'< D'$.

 We choose a vertex $u \in V_i$ and let $G '��$  be the graph obtained by deleting all edges incident to u and joining u to each vertex in $V_{j-1}\cup V_{j+1}$ of $G $. Clearly $G '�� \in \mathcal{B}$$(n, d)$. Hence we have
\begin{eqnarray}
 \nonumber  \frac{ \prod_1({G} )}{\prod_1({G'} )}=(\frac{C'��}{D'})^2<1. ��
\end{eqnarray}

 This completes the proof of Lemma 3.4.
\end{proof}

By the same method of Lemma 3.4, the following lemma is obtained analogously.

\begin{lemma}
 Let $G \in \mathcal{B}(n, d)$ with the maximal $\prod_2$-value. Then there exist at most two partition sets $V_i$ and $V_j$ such that $|V_i| \geq  2$, $|V_j| \geq 2$ and $|i-j| =1$ with $0\leq i, j \leq d$.
\end{lemma}

\begin{proof}
If $d=2$, the proof is done. If $d=3$, by $|V_0|=|V_d|=1$, the proof is also done. Now we consider the case of  $d \geq 4$. If there exists just one partition set (say $V_i$)  in $G$ such that $|V_i| \geq2$, then we are done. In order to show this lemma, it is enough to prove that if there exist at least two partition sets whose orders are greater than or equal to $2$, then we can deduce that for each pair of such partition sets (say $V_i $ and $V_j$) with $|V_i| \geq 2$ and $|V_j| \geq 2$, and one must have $|i - j| =1$.  Note that by Lemma 3.2, $G[V_{\ell-1}\cup V_{\ell}]$  is a complete bipartite subgraph for each $\ell \in  \{1, 2,\cdots , d\}$. We can obtain that each of the vertices in partition set $V_{\ell}$  has the same degree (say $d_{\ell}), ~\ell = 1, \cdots , d$.

Choose a graph $G$ in  $\mathcal{B}(n, d)$ {that achieves} the maximal value of the second multiplicative Zagreb index. Assume that there exist two partition sets $V_i$ and $V_j$ such that $|V_i| \geq  2, |V_j| \geq 2 $ and $|i-j|\geq 2$.

\noindent{\bf \small Case 1. } $m_{i+1}=m_{i+2}=m_{i+3}=\cdots =m_{j-2}= m_{j-1}=1$.

In order to proceed conveniently, we set that  

$ A_1= (m_{i-2}+m_i)^{(m_{i-2}+m_i)m_{i-1}}(1+m_{j+1})^{(1+m_{j+1})}(1+m_i)^{(1+m_i)}(1+m_j)^{(1+m_{j})}$

$~~~~~~~(m_j + m_{j+2})^{(m_j + m_{j+2})m_{j+1}}$,

 $ B_1=( m_{i-2}+m_i+1)^{(m_{i-2}+m_i+1)m_{i-1}}(1+m_{i-1})^{(1+m_{i-1})}(2+m_i)^{(2+m_i)}(m_j)^{m_{j}}$

$~~~~~~~(m_j + m_{j+2}-1)^{(m_j + m_{j+2}-1)m_{j+1}}$,

  $C_1=(m_{i-2}+m_i)^{(m_{i-2}+m_i)m_{i-1}}(1+m_{i-1})^{(1+m_{i-1})}(1+m_i)^{(1+m_i)}(1+m_j)^{(1+m_{j})}$

$~~~~~~~(m_j + m_{j+2})^{(m_j + m_{j+2})m_{j+1}}$,

  $D_1=( m_{i-2}+m_i-1)^{(m_{i-2}+m_i-1)m_{i-1}}(1+m_{j+1})^{(1+m_{j+1})}(m_i)^{m_i}(2+m_j)^{(2+m_{j})}$

$~~~~~~~(m_j + m_{j+2}+1)^{(m_j + m_{j+2}+1)m_{j+1}}$.

  \noindent{\bf \small Subcase 1.1. }   $A_1 < B_1$.

 We choose a vertex $u \in V_j$ and let $G '��$  be the graph obtained by deleting all edges incident to u and joining u to each vertex in $V_{i-1}\cup V_{i+1}$ of $G $. Clearly $G '�� \in \mathcal{B}$$(n, d)$. Hence we have
\begin{eqnarray}
 \nonumber  \frac{ \prod_2({G} )}{\prod_2({G'} )} &=&\frac{(m_{i-2}+m_i)^{(m_{i-2}+m_i)m_{i-1}}(1+m_{j+1})^{(1+m_{j+1})}(1+m_i)^{(1+m_i)}}{( m_{i-2}+m_i+1)^{(m_{i-2}+m_i+1)m_{i-1}}(1+m_{i-1})^{(1+m_{i-1})}(2+m_i)^{(2+m_i)}}
   \\ \nonumber &&\times \frac{(1+m_j)^{(1+m_{j})}(m_j + m_{j+2})^{(m_j + m_{j+2})m_{j+1}}}{(m_j)^{m_{j}}(m_j + m_{j+2}-1)^{(m_j + m_{j+2}-1)m_{j+1}}}
  \\ \nonumber &=&\frac{A_1}{B_1}<1.
\end{eqnarray}

  \noindent{\bf \small Subcase 1.2}  $A_1 \geq B_1$.��
 
 { We consider another claim.}

  \noindent{\bf Claim 3} If  $A_1\geq B_1$, then $D_1>C_1$.

\begin{proof} 
\begin{eqnarray}
 \nonumber D_1 &=&( m_{i-2}+m_i-1)^{(m_{i-2}+m_i-1)m_{i-1}}(1+m_{j+1})^{(1+m_{j+1})}(m_i)^{m_i}
 \\ \nonumber &&\;\;\;\; \;\; \times (2+m_j)^{(2+m_j)}(m_j + m_{j+2}+1)^{(m_j + m_{j+2}+1)m_{j+1}}
\\ \nonumber  &=&\frac{( m_{i-2}+m_i-1)^{(m_{i-2}+m_i-1)m_{i-1}}(m_i)^{m_i}(2+m_j)^{(2+m_{j})}}{  (m_{i-2}+m_i)^{(m_{i-2}+m_i)m_{i-1}}(1+m_i)^{(1+m_i)}(1+m_j)^{(1+m_{j})}}
\\ \nonumber  &&\;\;\;\;\;\; \times \frac{(m_j + m_{j+2}+1)^{(m_j + m_{j+2}+1)m_{j+1}}}{(m_j + m_{j+2})^{(m_j + m_{j+2})m_{j+1}}} \times A_1 \;\;\;\;\;\;\;(Note \;\;A_1 \;\;\geq B_1)
\\ \nonumber  &\geq & \left(\frac{ \frac{( m_{i-2}+m_i-1)^{(m_{i-2}+m_i-1)}}{( m_{i-2}+m_i)^{(m_{i-2}+m_i)}}} {\frac{( m_{i-2}+m_i)^{(m_{i-2}+m_i)}}{( m_{i-2}+m_i+1)^{(m_{i-2}+m_i+1)}}}\right)^{m_{i-1}} \times \left(\frac{ \frac{( m_{j}+m_{j+2}-1)^{(m_{j}+m_{j+2}-1)}}{( m_{j}+m_{j+2})^{(m_{j}+m_{j+2})}}} {\frac{( m_{j}+m_{j+2})^{(m_{j}+m_{j+2})}}{( m_{j}+m_{j+2}+1)^{(m_{j}+m_{j+2}+1)}}}\right)^{m_{j+1}}
\\ \nonumber &&\;\;\;\; \times \frac{ \frac{( m_{j})^{m_j}}{( m_{j}+1)^{(m_{j}+1)}}} {\frac{( m_{j}+1)^{(m_{j}+1)}}{( m_{j}+2)^{(m_{j}+2)}}}
\times  \frac{ \frac{( m_i)^{m_i}}{( m_{i }+1)^{(m_{i }+1)}}} {\frac{( m_{i 2}+1)^{(m_{i }+1)}}{( m_{i }+2)^{(m_{i }+2)}}} \times C_1
\\ \nonumber  &>& C_1 \;\;(by\;\; Proposition \;\;2.2).
\end{eqnarray}
\end{proof}

We choose a vertex $u \in V_i$ and let $G '��$  be the graph obtained by deleting all edges incident to u and joining u to each vertex in $V_{j-1}\cup V_{j+1}$ of $G $. Clearly $G '�� \in \mathcal{B}$$(n, d)$. Hence we have
\begin{eqnarray}
 \nonumber  \frac{ \prod_2({G} )}{\prod_2({G'} )} =\frac{C_1}{D_1}<1.��
\end{eqnarray}

\noindent{\bf \small Case 2. } $V_i$, $V_{i+1} $ and $V_{i+2}$ is successive three partitions such that  $|V_i|\geq 2$ , $|V_{i+1}|\geq 2 $ and $|V_{i+2}|\geq 2$, where $V_{j}=V_{i+2}$.  Here we let that 

     $ A_1'= (m_{i-2}+m_i)^{(m_{i-2}+m_i)m_{i-1}}(m_{i+1}+m_{j+1})^{(m_{i+1}+m_{j+1})}(m_j + m_{j+2})^{(m_j + m_{j+2})m_{j+1}}$,

 $ B_1'=( m_{i-2}+m_i+1)^{(m_{i-2}+m_i+1)m_{i-1}} (m_{i+1}+m_{i-1})^{(m_{i+1}+m_{i-1})} (m_j + m_{j+2}-1)^{(m_j + m_{j+2}-1)m_{j+1}}$,

  $C_1'=(m_{i-2}+m_i)^{(m_{i-2}+m_i)m_{i-1}}  (m_{i+1}+m_{i-1})^{(m_{i+1}+m_{i-1})}(m_j + m_{j+2})^{(m_j + m_{j+2})m_{j+1}}$,

  $D_1'=( m_{i-2}+m_i-1)^{(m_{i-2}+m_i-1)m_{i-1}} (m_{i+1}+m_{j+1})^{(m_{i+1}+m_{j+1})} (m_j + m_{j+2}+1)^{(m_j + m_{j+2}+1)m_{j+1}}$.

  \noindent{\bf \small Subcase 2.1.} $A_1'< B_1'$.

 We choose a vertex $u \in V_j$ and let $G '��$  be the graph obtained by deleting all edges incident to u and joining u to each vertex in $V_{i-1}\cup V_{i+1}$ of $G $. Clearly $G '�� \in \mathcal{B}$$(n, d)$. Hence we have
\begin{eqnarray}
 \nonumber  \frac{ \prod_2({G} )}{\prod_2({G'} )} = \frac{A_1'}{B_1'}  <1.
 \end{eqnarray}

\noindent{\bf \small Subcase 2.2.} $A_1' \geq B_1' $.��

 { We consider another claim.}

 \noindent{\bf Claim 4.} If  $A_1'\geq B_1'$, then $D_1'> C_1'$.

\begin{proof}
\begin{eqnarray}
 \nonumber D_1' &=&( m_{i-2}+m_i-1)^{(m_{i-2}+m_i-1)m_{i-1}}(m_{i+1}+m_{j+1})^{(m_{i+1}+m_{j+1})}
 \\ \nonumber &&\;\;\;\; \times (m_j + m_{j+2}+1)^{(m_j + m_{j+2}+1)m_{j+1}}
\\ \nonumber  &=&\frac{( m_{i-2}+m_i-1)^{(m_{i-2}+m_i-1)m_{i-1}}(m_j + m_{j+2}+1)^{(m_j + m_{j+2}+1)m_{j+1}}}{  (m_{i-2}+m_i)^{(m_{i-2}+m_i)m_{i-1}}(m_j + m_{j+2})^{(m_j + m_{j+2})m_{j+1}}} \times A_1'  \;\;(A_1' \geq B_1')
\\ \nonumber  &\geq &  \left(\frac{ \frac{( m_{i-2}+m_i-1)^{(m_{i-2}+m_i-1)}}{( m_{i-2}+m_i)^{(m_{i-2}+m_i)}}} {\frac{( m_{i-2}+m_i)^{(m_{i-2}+m_i)}}{( m_{i-2}+m_i+1)^{(m_{i-2}+m_i+1)}}}\right)^{m_{i-1}} \times \left(\frac{ \frac{( m_{j}+m_{j+2}-1)^{(m_{j}+m_{j+2}-1)}}{( m_{j}+m_{j+2})^{(m_{j}+m_{j+2})}}} {\frac{( m_{j}+m_{j+2})^{(m_{j}+m_{j+2})}}{( m_{j}+m_{j+2}+1)^{(m_{j}+m_{j+2}+1)}}}\right)^{m_{j+1}}  \times C_1'
\\ \nonumber  & >& C_1' \;\;\;\;\;(by\;\; Proposition \;\;2.2).
\end{eqnarray}
\end{proof}

We choose a vertex $u \in V_i$ and let $G '��$  be the graph obtained by deleting all edges incident to u and joining u to each vertex in $V_{j-1}\cup V_{j+1}$ of $G $. Clearly $G '�� \in \mathcal{B}$$(n, d)$. Hence we have
\begin{eqnarray}
 \nonumber  \frac{ \prod_2({G} )}{\prod_2({G'} )}=\frac{C_1'��}{D_1'}<1. ��
\end{eqnarray}

 This completes the proof of Lemma 3.5.
\end{proof}

Let $G \in \mathcal{B}(n, d)$  with the maximal value of the second multiplicative Zagreb index. In view of Lemma 3.4, assume that $|V_a| > 1$ and $|V_{a+1}| > 1$, and $|V_j| = 1$ for $j\in \{0, 1, \cdots , d\}-\{a, a+1\}$. By Lemma 3.2, any two consecutive partition sets induce a complete bipartite subgraph. Therefore, we can define $G$ by $G[a \cdot 1, s, t, b\cdot 1]$, where $s=m_a = |V_a|, t= m_{a+1} =|V_{a+1}|$, $a + b =d - 1$ and $s + t =n -d+ 1$.  In the whole  context we assume, without loss of generality, that $a \leq  b$ for a graph $G[a \cdot 1, s,t , b \cdot 1]$.

\begin{lemma}
Let $G[a \cdot 1, s, t, b\cdot 1] \in \mathcal{B}(n, d)$ be a graph with the maximal value of  $\prod_1$-value. Then $|s - t| \leq 1$.
\end{lemma}

\begin{proof}  If $d=2$, then the proof is straightforward. Now we suppose $d \geq  3$. According to the construction of the partition sets of G and by Lemma 3.1, we have $a \geq 1, b \geq 1$. Suppose $|s - t| \geq 2$. We assume without loss of generality that $t>s$, and then $t-s \geq 2$. Because $s = |V_a|$,  $t = |V_{a+1}|$, we have $d_a=t +1$,  $d_{a+1} = s+1$. Similarly, we have $d_{a-2}$,  $d_{a+3} \in \{  1, 2\}$. Thus $|d_{a-2}-d_{a+3}| \leq  2$.

Choose a vertex $u \in  V_{a+1}$ and let $G'$  be the graph obtained by deleting all edges incident to u and joining $u$ to each vertex
 in $(V_{a-1}\cup  V_{a+1})-\{u\}$. Clearly $G' \in $  $\mathcal{B}(n, d)$ and
  \begin{eqnarray}
 \nonumber  \frac{ \prod_1({G} )}{\prod_1({G'} )} &=&\frac{ (s+1)^2(t+1)^{2s}(s+1)^{2t}(t+1)^2}{(s+2)^2   t^{2(s+1)} (s+2)^{2(t-1)}t^2}
\\  \nonumber &=& \frac{\left(\frac{s+1}{s+2}\right)^{2t}}{  (\frac{ t }{ t+1  })^{2s+2}}\left(\frac{s+1}  {t}\right)^2
\\ \nonumber &= &\frac{(\frac{s+1}{s+2})^{2s+2}}{(\frac{t}{t+1})^{2s+2}}\left(\frac{s+1}{s+2}\right)^{2p+2}\left(\frac{s+1}{t}\right)^2
\\ \nonumber &&({Since}\; t\geq s+2, ~{ let} ~t=s+2+p \;and \; {by ~Proposition} \;2.1)
  \\ \nonumber &<&\left(\frac{s+1}{s+2}\right)^{2p+2}\left(\frac{s+1}{s+2+p}\right)^{2}
\\&<&1,
\end{eqnarray}

which is a contradiction. This completes the proof of Lemma 3.6.
\end{proof}

\begin{lemma}
Let $G[a \cdot1, s, t, b\cdot 1] \in$  $\mathcal{B}(n, d)$ be the graph with the maximal $\prod_2$-value. Then $|s -t| \leq 1$.
\end{lemma}

\begin{proof}   If $d=2$, then the proof is straightforward.  Now we suppose $d \geq  3$. According to the construction of the partition sets of G and by Lemma 2.1, we have $a \geq 1, b \geq 1$. Suppose $|s - t| \geq 2$; we assume without loss of generality that $t>s$, and then $t-s \geq 2$. Because $s = |V_a|$,  $t = |V_{a+1}|$, we have $d_a=t +1$,  $d_{a+1} = s+1$. Similarly, we have $d_{a-2}$,  $d_{a+3} \in \{ 1, 2\}$. Thus $|d_{a-2}-d_{a+3}| \leq  2$.

Choose a vertex $u \in  V_{a+1}$ and let $G'$  be the graph obtained by deleting all edges incident to u and joining $u$ to each vertex
 in $(V_{a-1}\cup  V_{a+1})-\{u\}$. Clearly $G' \in $  $\mathcal{B}(n, d)$ and
  \begin{eqnarray}
 \nonumber  \frac{ \prod_2({G} )}{\prod_2({G'} )} &=&\frac{(s+1)^{(s+1)}\cdot (t+1)^{s(t+1)} \cdot (s+1)^{t(s+1)}\cdot (t+1)^{t+1}}{(s+2)^{s+2}\cdot   t^{t(s+1)}\cdot (s+2)^{(s+2)(t-1)}\cdot  t^{t}}
\\  \nonumber &=&\frac{((s+1)(t+1))^{(s+1)(t+1)}}{(t(s+2)])^{t(s+2)}}
  \\  \nonumber&&     (\text{Let} \; t=s+2+p )
  \\  \nonumber& =& \frac{(st+2s+p+3)^{(st+2s+p+3)}}{(st+2s+2p+4)^{(st+2s+2p+4)}}
\\&<&1,
\end{eqnarray}
which is a contradiction. This completes the proof of Lemma 3.7.
\end{proof}

By the above lemmas and   routine calculations, one can
derive that

\begin{thom}
Let $G \in  \mathcal{B}(n, d)$ with the maximal $\prod_1$-value or $\prod_2$-value. Then $G\cong G[a\cdot 1, \lfloor \frac{n-d+1}{2} \rfloor $, $\lceil   \frac{n-d+1}{2} \rceil, b\cdot 1]$.  Furthermore $a, b$ satisfy the following conditions with respect to the diameter $d$ of $G$.

(i)     If $d =2$,  then $a =0,  b=1$;

(ii)    If $d=3$, then $a=1, b=1$;

(iii)   If $d=4$,  then $a=1, b=2$;

(iv)    If $d=5$,  then $a=2, b=2$;

(v) If $d=6$, then $a=2, b=3$;

(vi)    If $d\geq 7$, then $a\geq 3, b\geq 3$.
\end{thom}

\section{  Ordering the extremal graphs according to their diameters }
  In this section, we investigate the relationship between $\prod_i(G[a\cdot 1,  \lfloor \frac{n-d+1}{2} \rfloor , \lceil   \frac{n-d+1}{2} \rceil, b\cdot 1])$ and $d$ for $i=1, 2$. As an application, we characterize the bipartite graphs with the largest, second-largest and smallest $\prod_1$-value (resp., $\prod_2$-values).

  \begin{thom} For $2 \leq d \leq n-1,  \prod_1(G[a\cdot 1,  \lfloor \frac{n-d+1}{2} \rfloor , \lceil   \frac{n-d+1}{2} \rceil, b\cdot 1])$ (resp.
  $\prod_2(G[a\cdot 1,  \lfloor \frac{n-d+1}{2} \rfloor , \lceil   \frac{n-d+1}{2} \rceil, b\cdot 1])$ is a decreasing function on $d$.
  \end{thom}

\begin{proof}    Let $G_d = G[a\cdot 1,  \lfloor \frac{n-d+1}{2} \rfloor , \lceil   \frac{n-d+1}{2} \rceil, b\cdot 1])$. Put $f(d)=\prod_1(G_d)$, $g(d)=\prod_2(G_d), d=2,3, \cdots , n-1$. In order to complete the proof  of this theorem, it suffices to prove the following claims.

\noindent{\bf   Claim 5.} $f(n-1) < f(n-2) < \cdots < f(8) < f(7)$ and  $g(n-1) <g(n-2) < \cdots < g(6) < g(5)$.

\noindent{\em Proof of Claim 5.}   Note that $d\leq n-1$.  Hence we have $\lceil   \frac{n-d}{2}  \rceil  \geq 1$.  For $d\geq 7$, $f(d)=4^{a+b-4}(s+1)^{2t+2}(t+1)^{2s+2}$, where $s=\lfloor \frac{n-d+1}{2} \rfloor $, $t=\lceil   \frac{n-d+1}{2} \rceil$, $a+b=d-1$,$s+t=n-d+1$.
Hence $f(d)=4^{d-5}(s+1)^{2t+2}(t+1)^{2s+2}$.
We have the following two cases:

\noindent {\bf \small  Case 1.} $n,d$ are even or $n,d$ are  odd. 

Then $n-d \geq 2$, and  an elementary calculation yields
$$f(d)=4^{d-5}(\frac{n-d+2}{2})^{n-d+4}(\frac{n-d+4}{2})^{n-d+2}.$$
Hence,
\begin{eqnarray}
 \nonumber  \frac{f(d+1)}{f(d)} &=&\frac{4^{d-4} (\frac{n-d+1}{2})^{n-d+3}(\frac{n-d+3}{2})^{n-d+1}}{ 4^{d-5} (\frac{n-d+2}{2})^{n-d+4}(\frac{n-d+4}{2})^{n-d+2}}
\\  \nonumber &= & 4(\frac{\frac{n-d+1}{2}}{\frac{n-d+2}{2}})^{n-d+3}(\frac{\frac{n-d+3}{2}}{\frac{n-d+4}{2}})^{n-d+1}(\frac{2}{n-d+2})(\frac{2}{n-d+4})
\\ \nonumber &< &\frac{16}{(n-d+2)(n-d+4)}\;\;\;\;\;(Since\;   n-d\geq 1)
\\&<&1.
\end{eqnarray}

\noindent {\bf \small  Case 2.}  $n$ is even and $d$ is odd or $n$  is odd and $d$ is even.

Then $n-d \geq 3$. Otherwise $n-d=1$, there is no graph of diameter $d+1$ with $n$ vertices.  An elementary calculation yields
$$f(d)=4^{d-5}(\frac{n-d+1}{2})^{2(n-d+3)}.$$
Hence,
\begin{eqnarray}
 \nonumber  \frac{f(d+1)}{f(d)}&=&\frac{4^{d-4} (\frac{n-d}{2})^{2(n-d+2)}}{ 4^{d-5} (\frac{n-d+1}{2})^{2(n-d+3)}}
\\  \nonumber &=& 4\left(\frac{\frac{n-d}{2}}{\frac{n-d+1}{2}}\right)^{2(n-d+2)}\left(\frac{4}{(n-d+1)^2}\right)
\\ \nonumber &<& \frac{16}{(n-d+1)^2} \;\;\;\;\;\;(Since \; n-d \geq 3 )
\\&<&1.
\end{eqnarray}

Similarly, we can also show that for $d\geq 5$, $\frac{g(d+1)}{g(d)}<1$.

This completes the proof of Claim 5.

\noindent {\bf \small  Claim 6.} $f(7)< f(6)<f(5) <f(4) <f(3)<f(2)$ and $g(5)<g(4)<g(3)<g(2)$.

\noindent{\em Proof of Claim 6.}  With a similar method we can also prove this  part by  direct computations.
The proof of Claim 6 is finished.

By Claims 5 and 6, we complete the proof of Theorem 4.1.
\end{proof}

The following corollary is a direct consequence of Lemma 3.8 and Theorem 4.1.

\begin{coro}  Among all bipartite graphs with order $n\geq 2$, $K_{\lfloor \frac{n}{2} \rfloor, \lceil  \frac{n}{2} \rceil}$
has the largest $\prod_1$-values and $\prod_2$-values, whereas $P_n$ has the smallest $\prod_1$-values and  $\prod_2$-values.
\end{coro}

\begin{thom} Among all bipartite graphs with order $n > 2$,  $K_{\lfloor \frac{n}{2} \rfloor, \lceil  \frac{n}{2} \rceil} -e$ has the second-largest $\prod_1$-values and $\prod_2$-values for odd $n$,  and $K_{\lfloor \frac{n}{2} \rfloor, \lceil  \frac{n}{2} \rceil} -e$ has the second-largest $\prod_1$-values for even $n$  and $K_{\lfloor \frac{n-2}{2} \rfloor, \lceil  \frac{n+2}{2} \rceil} $ has the second-largest $\prod_2$-values for even $n$.
\end{thom}

\begin{proof}
 Note that $\prod_1(K_{s, t})=t^{2s}s^{2t}$ and $\prod_2(K_{s,t})=(st)^{st}$.   By Theorem 4.1, we only need to compare the value of the second multiplicative Zagreb indices of  $K_{\lfloor \frac{n-2}{2} \rfloor, \lceil  \frac{n+2}{2} \rceil}$ and  $G[1, \lfloor \frac{n-2}{2} \rfloor, \lceil  \frac{n-2}{2} \rceil , 1 ]$. Note that $G[1, \lfloor \frac{n-2}{2} \rfloor, \lceil  \frac{n-2}{2} \rceil , 1 ]$
$\cong $ $K_{\lfloor \frac{n}{2} \rfloor, \lceil  \frac{n}{2} \rceil} -e$.

By direct computations, we have that
$$\prod_1( K_{\lfloor \frac{n-2}{2} \rfloor, \lceil  \frac{n+2}{2}
\rceil})=t^{2s}s^{2t}=(\lceil \frac{n+2}{2} \rceil
)^{2(\lfloor \frac{n-2}{2} \rfloor )} \cdot (\lfloor \frac{n-2}{2}
\rfloor)^{2(\lceil  \frac{n+2}{2} \rceil)}$$
and
$$\prod_2(K_{\lfloor \frac{n-2}{2} \rfloor, \lceil  \frac{n+2}{2}
\rceil}) =(ts)^{ts}=(\lfloor \frac{n-2}{2} \rfloor \cdot \lceil
\frac{n+2}{2} \rceil )^{\lfloor \frac{n-2}{2} \rfloor \cdot \lceil
\frac{n+2}{2} \rceil }.$$

Since $G[1, \lfloor \frac{n-2}{2} \rfloor, \lceil  \frac{n-2}{2} \rceil , 1 ]$
$\cong $ $K_{\lfloor \frac{n}{2} \rfloor, \lceil  \frac{n}{2} \rceil} -e$,
we have that
$$\prod_1(K_{\lfloor \frac{n}{2} \rfloor, \lceil  \frac{n}{2} \rceil}-e)
 =(\lfloor \frac{n-2}{2} \rfloor +1)^{2(\lceil  \frac{n-2}{2} \rceil  +1)}\cdot (\lceil  \frac{n-2}{2} \rceil  +1)^{2(\lfloor \frac{n-2}{2} \rfloor +1)}$$
and
$$\prod_2(K_{\lfloor \frac{n}{2} \rfloor, \lceil  \frac{n}{2} \rceil}-e)=
  [(\lfloor \frac{n-2}{2} \rfloor +1)(\lceil  \frac{n-2}{2} \rceil  +1)]^{(\lfloor \frac{n-2}{2} \rfloor +1)(\lceil  \frac{n-2}{2} \rceil  +1)}.$$

If $n$ is even, by direct calculations, we obtain that

$$\prod_1( K_{\lfloor \frac{n-2}{2} \rfloor, \lceil  \frac{n+2}{2} \rceil})=(\frac{n}{2}+1)^{(n-2)}\cdot (\frac{n}{2}-1)^{(n+2)}, $$
$$\prod_2(K_{\lfloor \frac{n-2}{2} \rfloor, \lceil  \frac{n+2}{2} \rceil})=(\frac{n-2}{2}\cdot \frac{n+2}{2})^{(\frac{n-2}{2}\cdot \frac{n+2}{2})},$$
$$\prod_1(K_{\lfloor \frac{n}{2} \rfloor, \lceil  \frac{n}{2} \rceil}-e)=(\frac{n}{2})^{n}\cdot (\frac{n}{2})^{n}$$
and
$$\prod_2(K_{\lfloor \frac{n}{2} \rfloor, \lceil  \frac{n}{2} \rceil}-e)=[\frac{n}{2}\cdot \frac{n}{2}]^{(\frac{n}{2}\cdot \frac{n}{2})}.$$

By direct calculations, we obtain

$$\frac{\prod_1( K_{\lfloor \frac{n-2}{2} \rfloor, \lceil  \frac{n+2}{2} \rceil})}{\prod_1(K_{\lfloor \frac{n}{2} \rfloor, \lceil  \frac{n}{2} \rceil}-e)} =\left(\frac{\frac{n}{2}-1}{\frac{n}{2}+1}\right)^2 \times \left(\frac{\left(\frac{\frac{n}{2}-1}{\frac{n}{2}}\right)}{\left(\frac{\frac{n}{2}}{\frac{n}{2}+1}\right)}\right)^n <1 \;\;\; (by ~Proposition ~2.1)$$

and
\begin{eqnarray}
 \nonumber  \frac{\prod_2( K_{\lfloor \frac{n-2}{2} \rfloor, \lceil  \frac{n+2}{2} \rceil})}{\prod_2(K_{\lfloor \frac{n}{2} \rfloor, \lceil  \frac{n}{2} \rceil}-e)} &=&\frac{((\frac{n}{2}+1)(\frac{n}{2}-1))^{(\frac{n}{2}+1)(n/2-1)}}{((\frac{n}{2})(\frac{n}{2}))^{(\frac{n}{2})(\frac{n}{2})}}
\\  \nonumber &=& \left(\frac{\left(\frac{(\frac{n}{2}-1)^{(\frac{n}{2}-1)}}{(\frac{n}{2})^{(\frac{n}{2})}}\right)}{\left(\frac{(\frac{n}{2})^{(\frac{n}{2})}}
{(\frac{n}{2}+1)^{(\frac{n}{2}+1)}}\right)}\right)^{(\frac{n}{2})} \times \frac{(\frac{n}{2}-1)^{(\frac{n}{2}-1)}}{(\frac{n}{2}+1)^{(\frac{n}{2}+1)}}
\\ \nonumber & >&1 \;\;\;\;\;(by \; Proposition \; 2.2 ).
\end{eqnarray}

If $n$ is odd, by direct calculations, we get that
$$\prod_1( K_{\lfloor \frac{n-2}{2} \rfloor, \lceil  \frac{n+2}{2} \rceil})=(\frac{n+3}{2})^{(n-3)}\cdot (\frac{n-3}{2})^{(n+3)}, $$
$$\prod_2(K_{\lfloor \frac{n-2}{2} \rfloor, \lceil  \frac{n+2}{2} \rceil})=(\frac{n-3}{2}\cdot \frac{n+3}{2})^{(\frac{n-3}{2}\cdot \frac{n+3}{2})},$$
$$\prod_1(K_{\lfloor \frac{n}{2} \rfloor, \lceil  \frac{n}{2} \rceil}-e)=(\frac{n-1}{2})^{n+1}\cdot (\frac{n+1}{2})^{(n-1)}$$
and
$$\prod_2(K_{\lfloor \frac{n}{2} \rfloor, \lceil  \frac{n}{2} \rceil}-e)=(\frac{n-1}{2}\cdot \frac{n+1}{2})^{(\frac{n-1}{2}\cdot \frac{n+1}{2})}.$$

By direct calculations, we obtain that

$$\frac{\prod_1( K_{\lfloor \frac{n-2}{2} \rfloor, \lceil
\frac{n+2}{2} \rceil})}{\prod_1(K_{\lfloor \frac{n}{2} \rfloor,
\lceil  \frac{n}{2} \rceil}-e)} = \left(\frac{\left(\frac{\frac{n-3}{2}}{\frac{n-1}{2}}\right)}{\left(\frac{\frac{n+1}{2}}{\frac{n+3}{2}}\right)}\right)^{n-3}
 \times \left(\frac{\frac{n-3}{2}}{\frac{n+1}{2}}\right)^2 \times \left(\frac{\frac{n-3}{2}}{\frac{n-1}{2}}\right)^4   <1 ~~(by~Proposition \;\;2.1)$$
 
and

$$\frac{\prod_2(
K_{\lfloor \frac{n-2}{2} \rfloor, \lceil  \frac{n+2}{2}
\rceil})}{\prod_2(K_{\lfloor \frac{n}{2} \rfloor, \lceil
\frac{n}{2} \rceil}-e)} =\frac{(\frac{n^2-9}{4})^{ \frac{n^2-9}{4}} }{(\frac{n^2-1}{4})^{ \frac{n^2-1}{4}} }<1.$$

 This completes the proof of Theorem 4.3.
\end{proof}

\vskip4mm\noindent{\bf Acknowledgements.}


 The work was partially supported by the National Natural Science Foundation of China under Grants 11371162, 11571134, 11601006,
 and the Self-determined Research Funds of CCNU from the colleges basic research and operation of MOE, and the Natural Science Foundation of Anhui Province of China
under Grant no. KJ2013B105, the Natural Science Foundation for the
Higher Education Institutions of Anhui Province of China under
Grant no. KJ2015A331. Furthermore, the authors are grateful to the anonymous referee for a careful checking of the details and for helpful comments that improved this paper.

\end{document}